\documentclass[11pt]{amsart}
\usepackage[margin=1in]{geometry}
\usepackage{latexsym}
\usepackage{amsfonts}
\usepackage{amsmath}
\usepackage{amssymb}
\usepackage{amsthm}
\usepackage{enumerate}
\usepackage[hang,flushmargin]{footmisc}
\usepackage{caption}
\usepackage{tabu}
\usepackage{mathrsfs}
\usepackage{amsaddr}
\usepackage{subfig}

\newcommand{\Av}{\operatorname{Av}}

\usepackage{graphicx}
\usepackage{epstopdf}
\usepackage{epsfig}
\usepackage{caption}
 
\usepackage{bm} 
\usepackage{cite}

\makeatletter
\newtheorem*{rep@theorem}{\rep@title}
\newcommand{\newreptheorem}[2]{%
\newenvironment{rep#1}[1]{%
 \def\rep@title{#2 \ref{##1}}%
 \begin{rep@theorem}}%
 {\end{rep@theorem}}}
\makeatother

\newtheorem{theorem}{Theorem}[section]
\newreptheorem{theorem}{Conjecture}
\newtheorem{lemma}{Lemma}[section]
\newtheorem{proposition}{Proposition}[section]
\newtheorem{corollary}{Corollary}[section]

\theoremstyle{definition}
\newtheorem{definition}{Definition}[section]
\newtheorem{remark}{Remark}[section]

\DeclareMathOperator{\de}{def}
\DeclareMathOperator{\fer}{fer}
\DeclareMathOperator{\swap}{swap}
\DeclareMathOperator{\slmax}{slmax}
\DeclareMathOperator{\indmax}{indmax}
\DeclareMathOperator{\lmax}{lmax}
\DeclareMathOperator{\LenDes}{LenDes}
\DeclareMathOperator{\swd}{swd}
\DeclareMathOperator{\swu}{swu}

\DeclareMathOperator{\rev}{rev}
\DeclareMathOperator{\Des}{Des}
\DeclareMathOperator{\des}{des}

\DeclareMathOperator{\peak}{peak}
\DeclareMathOperator{\tl}{tl}

\DeclareMathOperator{\zeil}{zeil}
\DeclareMathOperator{\rmax}{rmax}

\DeclareMathOperator{\DPT}{\mathsf{DPT}}

\begin{document}
\title{Polyurethane Toggles}
\author{Colin Defant}
\address{Princeton University \\ Fine Hall, 304 Washington Rd. \\ Princeton, NJ 08544}
\email{cdefant@princeton.edu}

\begin{abstract}
We consider the involutions known as \emph{toggles}, which have been used to give simplified proofs of the fundamental properties of the promotion and evacuation maps. We transfer these involutions so that they generate a group $\mathscr P_n$ that acts on the set $S_n$ of permutations of $\{1,\ldots,n\}$. After characterizing its orbits in terms of permutation skeletons, we apply the action in order to understand West's stack-sorting map. We obtain a very simple proof of a result that clarifies and extensively generalizes a theorem of Bouvel and Guibert and also generalizes a theorem of Bousquet-M\'elou. We also settle a conjecture of Bouvel and Guibert. We prove a result related to the recently-introduced notion of postorder Wilf equivalence. Finally, we investigate an interesting connection among the action of $\mathscr P_n$ on $S_n$, the group structure of $S_n$, and the stack-sorting map. 
\end{abstract}

\maketitle

\bigskip

\section{Introduction}\label{Sec:Intro}

\subsection{Toggles, Trees, and Permutations}

A \emph{linear extension} of an $n$-element poset $P$ is a bijection $L:P\to[n]$ such that $L(x)\leq L(y)$ whenever $x\leq_{P}y$. We often view $L$ as a labeling of the elements of $P$, where $L(x)$ is the label of $x$. Let $\mathcal L(P)$ denote the set of linear extensions of $P$. In \cite{Bender}, Bender and Knuth made use of special involutions on semistandard Young tableaux, which, when restricted to standard Young tableaux, can be seen as involutions on the set of linear extensions of a poset. Subsequently, these have been called \emph{Bender-Knuth involutions}. Promotion and evacuation are bijections defined on $\mathcal L(P)$ that were first studied extensively by Sch\"utzenberger \cite{Schutzenberger1,Schutzenberger2,
Schutzenberger3}. We refer the reader to Stanley's beautiful survey article \cite{StanleyPromotion}, which gives much more information about these important maps. Haiman \cite{Haiman} and Malvenuto--Reutenauer \cite{Malvenuto} simplified Sch\"utzenberger's approach by showing that promotion and evacuation can be defined in terms of generalizations of the Bender-Knuth involutions. Following the work of Striker and Williams in \cite{Striker}, the term \emph{toggle} has been used to refer to Bender-Knuth involutions and other related types of involutions. Roughly speaking, a toggle is an involution on a set of combinatorial objects that only makes a small local change. See \cite{Roby} for more information about toggles. 

The toggles that Haiman and Malvenuto--Reutenauer used are defined as follows. If $i\in[n-1]$ and $L\in\mathcal L(P)$, then we obtain a new labeling $\swap_i(L)$ of $P$ by swapping the labels $i$ and $i+1$. Note that $\swap_i(L)$ is a linear extension of $P$ if and only if the elements $L^{-1}(i)$ and $L^{-1}(i+1)$ are incomparable in $P$. The toggle $p_i:\mathcal L(P)\to \mathcal L(P)$ is defined by \[p_i(L)=\begin{cases} \swap_i(L), & \mbox{if } \swap_i(L)\in\mathcal L( P); \\ L, & \mbox{otherwise.}  \end{cases}\] Note that each map $p_i$ is an involution and that $p_i\circ p_j=p_j\circ p_i$ whenever $i$ and $j$ are not consecutive integers. We let $\mathfrak S_Z$ denote the symmetric group on a set $Z$, which is the group of bijections from $Z$ to itself. Thus, $p_1,\ldots,p_{n-1}\in\mathfrak S_{\mathcal L( P)}$. The \emph{toggle group of} $p_1,\ldots,p_{n-1}$ is the subgroup of $\mathfrak S_{\mathcal L( P)}$ generated by $p_1,\ldots,p_{n-1}$. 

A \emph{rooted plane tree} is a rooted tree with finitely many vertices in which the (possibly empty) subtrees of each vertex are linearly ordered (from left to right). Binary plane trees, ternary plane trees, Motzkin trees, and many other natural trees are all examples of rooted plane trees. Given a set $X$ of positive integers, a \emph{decreasing plane tree on $X$} is a rooted plane tree whose vertices are bijectively labeled with the elements of $X$ so that every nonroot vertex has a label that is smaller than the label of its parent. If $X=[n]$, then this is the same as a linear extension of the poset whose Hasse diagram is the rooted plane tree (where the root is the maximum element and the leaves are the minimal elements). Let $\DPT$ be the set of decreasing plane trees. The \emph{skeleton} of a decreasing plane tree $T$ is the rooted plane tree obtained by removing the labels from $T$. A \emph{binary plane tree} is a rooted plane tree in which each vertex has exactly $2$ (possibly empty) subtrees (the left and right subtrees). Let $\mathsf{DBPT}\subseteq\DPT$ be the set of decreasing binary plane trees.  

Throughout this article, unless otherwise specified, we use the word \emph{permutation} to refer to an ordering of a finite set of positive integers, written as a word. Let $S_n$ denote the set of permutations of $[n]$. We can obtain a permutation from a labeled tree using a \emph{tree traversal}. One useful tree traversal that is defined on decreasing binary plane trees is the \emph{in-order} traversal (sometimes called the \emph{symmetric order} traversal). In order to read a decreasing binary plane tree in in-order, we read the left subtree of the root in in-order, then read the label of the root, and finally read the right subtree of the root in in-order. For example, \[\begin{array}{l}\includegraphics[height=1.3cm]{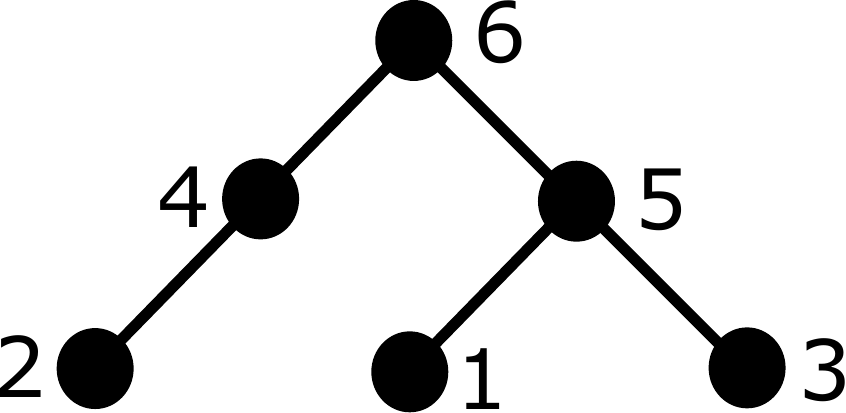}\end{array}\xrightarrow{\,\,\,\mathcal I\,\,\,}246153.\]
The in-order reading $\mathcal I(T)$ of a decreasing binary plane tree $T$ is a permutation of the set of labels of $T$. It is well known \cite{Bona, Stanley} that the in-order reading $\mathcal I$ is a bijection from $\mathsf{DBPT}$ to the set of all permutations. Since we have already defined the skeleton of a decreasing binary plane tree (this is the tree obtained by removing the labels), it makes sense to define skeletons of permutations. Namely, the skeleton of a permutation $\pi$ is the skeleton of $\mathcal I^{-1}(\pi)$.

Because each permutation $\pi\in S_n$ has an associated decreasing binary plane tree $\mathcal I^{-1}(\pi)$, which can be seen as a linear extension of a poset, we can transfer the toggles $p_1,\ldots,p_{n-1}$ above to obtain toggles $p_1,\ldots,p_{n-1}\in\mathfrak S_{S_n}$.\footnote{We use the same symbols $p_1,\ldots,p_{n-1}$ by an abuse of notation, but this should not lead to any confusion since we usually only consider the toggles defined on $S_n$.} Thus, we obtain a group $\mathscr P_n=\langle p_1,\ldots,p_{n-1}\rangle\leq\mathfrak S_{S_n}$. Note that $\mathscr P_n$ is not a toggle group as we defined it above because different permutations could have different skeletons that then give rise to different posets. The permutations $\pi$ and $p_i(\pi)$ always have the same skeleton, so all of the elements of $\mathscr P_n$ preserve skeletons. We call the maps $p_1,\ldots,p_{n-1}$ \emph{polyurethane toggles} and call $\mathscr P_n$ the $n^\text{th}$ \emph{polyurethane group}.\footnote{Polyurethane is a polymer used to manufacture surface coatings that preserve skeletons.} 

Just as Haiman and Malvenuto--Reutenauer used toggles to simplify Sch\"utzenberger's proofs concerning promotion and evacuation, we will use the above toggles on $S_n$ to generalize and simplify the proofs of several results concerning West's stack-sorting map. This is a function $s$ that sends permutations to permutations; we define it in Section \ref{Sec:Preliminaries}. 

Although we will not need this fact, we wish to remark that the polyurethane toggles fit into a more general context explored by Bj\"orner and Wachs \cite{Bjorner}. Given $\pi=\pi_1\cdots\pi_n\in S_n$, let $\pi^{-1}$ be the permutation whose $\pi_i^\text{th}$ entry is $i$ for all $i\in[n]$. Fix a binary plane tree $T$, and let $A$ be the set of permutations in $S_n$ with skeleton $T$. A special consequence of one of the main results in \cite{Bjorner} (written using different language) is that the set $s(A)^{-1}=\{s(\pi)^{-1}:\pi\in A\}$ is an interval in the weak Bruhat order and that the inversion statistic and the major index are equidistributed on $s(A)^{-1}$.

\subsection{Summary of Main Results}

Section \ref{Sec:Preliminaries} provides necessary background on the stack-sorting map, the postorder traversal, and permutation statistics. We begin Section \ref{Sec:Main} with a simple proof that two permutations in $S_n$ are in the same orbit under the action of the polyurethane group $\mathscr P_n$ if and only if they have the same skeleton. We will also see that every orbit contains a unique $231$-avoiding permutation and a unique $132$-avoiding permutation. This allows us to give useful alternative definitions of two ``sliding operators" that were defined and used heavily in \cite{DefantCatalan} and \cite{DefantFertilityWilf}. We then give a new proof of one of the main theorems from \cite{Bouvel} that is much simpler than the original proof and neatly explains why the permutation statistics appearing in that theorem are actually there. Our proof yields a result that is much stronger than the original theorem, and it allows us to prove a conjecture of Bouvel and Guibert as a simple consequence. 

We prove that if $\pi,\pi'\in S_n$ have the same skeleton, then there is a skeleton-preserving bijection $\omega:\mathcal P^{-1}(\pi)\to \mathcal P^{-1}(\pi')$, where $\mathcal P$ is the postorder traversal defined in Section \ref{Sec:Preliminaries}. We then consider a theorem of Bousquet-M\'elou concerning ``sorted permutations" and ``canonical preimages." We give a vast generalization of her result in terms of what we call ``higher-order twisted stack-sorting operators." In fact, this generalization simultaneously subsumes Bousquet-M\'elou's theorem and the aforementioned conjecture of Bouvel and Guibert. We also show that the skeleton of a permutation in $S_n$ determines the skeleton of its canonical preimage (see Section \ref{Sec:Main} for definitions).

Finally, we will consider an interesting connection among the action of $\mathscr P_n$ on $S_n$, the group structure of $S_n$, and the stack-sorting map. This allows us to give a new description of one of the sliding operators from \cite{DefantCatalan} and \cite{DefantFertilityWilf}. We end with two open problems, one of which is the problem of determining the isomorphism type of the polyurethane group $\mathscr P_n$. 

\section{Preliminaries}\label{Sec:Preliminaries}

If $\pi$ is a permutation of a set of $n$ positive integers, then the \emph{normalization} (also call the \emph{standardization}) of $\pi$ is the permutation in $S_n$ obtained by replacing the $i^\text{th}$-smallest entry in $\pi$ with $i$ for all $i\in[n]$. We say a permutation is \emph{normalized} if it is equal to its normalization. Given $\tau\in S_m$, we say a permutation $\sigma=\sigma_1\cdots\sigma_n$ \emph{contains the pattern} $\tau$ if there exist indices $i_1<\cdots<i_m$ in $[n]$ such that the normalization of $\sigma_{i_1}\cdots\sigma_{i_m}$ is $\tau$. We say $\sigma$ \emph{avoids} $\tau$ if it does not contain $\tau$. Let $\Av(\tau^{(1)},\tau^{(2)},\ldots)$ denote the set of normalized permutations that avoid the patterns $\tau^{(1)},\tau^{(2)},\ldots$ (this list of patterns could be finite or infinite). A set of the form $\Av(\tau^{(1)},\tau^{(2)},\ldots)$ is called a \emph{permutation class}. Let $\Av_n(\tau^{(1)},\tau^{(2)},\ldots)=\Av(\tau^{(1)},\tau^{(2)},\ldots)\cap S_n$.  

In his seminal monograph \emph{The Art of Computer Programming}, Knuth \cite{Knuth} defined a certain \emph{stack-sorting algorithm}. His analysis of this algorithm led to several important advances in combinatorics and theoretical computer science, such as the \emph{kernel method} \cite{Banderier} and the notion of permutation pattern avoidance \cite{Bona,Kitaev}. In his dissertation, West \cite{West} defined a deterministic variant of Knuth's algorithm, which has now received a large amount of attention \cite{Bona, BonaWords, BonaSurvey, BonaSimplicial, BonaSymmetry, Bousquet98, Bousquet, Bouvel, BrandenActions, Branden3, Claesson, Cori, DefantCatalan, DefantCounting, DefantDescents, DefantEnumeration, DefantFertility, DefantFertilityWilf, DefantPostorder, DefantPreimages, DefantClass, DefantEngenMiller, DefantKravitz, Dulucq, Dulucq2, Fang, Goulden, Ulfarsson, West, Zeilberger}. This variant is a function $s$, called the \emph{stack-sorting map}. The function $s$ sends the empty permutation to itself. If $\pi$ is a nonempty permutation with largest entry $n$, then we can write $\pi=LnR$. We then define $s$ recursively by $s(\pi)=s(L)s(R)n$. For example, 
\begin{equation}\label{Eq5}
s(246153)=s(24)\,s(153)\,6=s(2)\,4\,s(1)\,s(3)\,5\,6=241356.
\end{equation}

Almost all questions that people have asked about the stack-sorting map can be phrased naturally in terms of preimages of permutations. West \cite{West} defined the \emph{fertility} of a permutation $\pi$ to be $|s^{-1}(\pi)|$. It follows from Knuth's analysis that the fertility of the identity permutation $123\cdots n$ is the $n^\text{th}$ Catalan number $C_n=\frac{1}{n+1}{2n\choose n}$. Indeed, Knuth showed that 
\begin{equation}\label{Eq1}
s^{-1}(123\cdots n)=\Av_n(231)\quad\text{and}\quad |\Av_n(231)|=C_n.
\end{equation} West also went through a great deal of effort to compute the fertilities of the permutations \[23\cdots k1(k+1)\cdots n,\quad 12\cdots(k-2)k(k-1)(k+1)\cdots n,\quad\text{and}\quad k12\cdots(k-1)(k+1)\cdots n,\] showing in particular that these fertilities are the same. Bousquet-M\'elou reproved the fact that the first and last permutations have the same fertility in \cite{Bousquet}, and the current author generalized this in \cite{DefantFertilityWilf} and \cite{DefantClass}.  

The articles \cite{Bouvel,Claesson,DefantClass,
DefantEnumeration,DefantFertilityWilf,DefantCounting} are concerned with stack-sorting preimages of permutation classes. One motivation for this line of work comes from the fact that $s^{-1}(\Av(231))$ is the set of $2$-stack-sortable permutations (see \cite{Bona,Bouvel,DefantCounting} for definitions). Another motivation comes from the fact that there are several permutation classes $\Av(\tau^{(1)},\tau^{(2)},\ldots)$ such that $s^{-1}(\Av(\tau^{(1)},\tau^{(2)},\ldots))$ is also a permutation class (see \cite{DefantClass,DefantEnumeration} for examples). The articles \cite{Bouvel,DefantClass,
DefantFertilityWilf} consider when the preimage sets of two permutation classes are counted by the same numbers, a phenomenon dating back to West that was named ``fertility Wilf equivalence" in \cite{DefantFertilityWilf}. These articles also consider permutation statistics that are jointly equidistributed on the various preimage sets. We will see that the action of $\mathscr P_n$ on $S_n$ yields a remarkably simple tool for analyzing fertility Wilf equivalence. 

We defined the in-order tree traversal $\mathcal I$ in the introduction. Another tree traversal, called the \emph{postorder traversal}, is defined on all decreasing plane trees. We read a decreasing plane tree in postorder by reading the subtrees of the root from left to right (each in postorder) and then reading the label of the root. Letting $\mathcal P(T)$ denote the postorder reading of a decreasing plane tree $T$, we find that $\mathcal P$ is a map from $\DPT$ to the set of all permutations. The fundamental link between the stack-sorting map and decreasing plane trees comes from the identity \cite{Bona}
\begin{equation}\label{Eq6}
s=\mathcal P\circ \mathcal I^{-1}.
\end{equation} 
For example, we have \[246153\xrightarrow{\mathcal I^{-1}}\begin{array}{l}\includegraphics[height=1.3cm]{SkeletonPIC3}\end{array}\xrightarrow{\,\,\,\mathcal P\,\,\,}241356,\] which agrees with \eqref{Eq5}. 

If $\mathscr T,\mathscr T'\subseteq\DPT$, then we say a map $\psi:\mathscr T\to\mathscr T'$ is \emph{skeleton-preserving} if $T$ and $\psi(T)$ have the same skeleton for all $T\in\mathscr T$. The article \cite{DefantFertilityWilf} considers when two permutation classes $\Av(\tau^{(1)},\tau^{(2)},\ldots)$ and $\Av(\tau'^{(1)},\tau'^{(2)},\ldots)$ are \emph{postorder Wilf equivalent}, which means that there exists a skeleton-preserving bijection \[\eta:\mathcal P^{-1}(\Av(\tau^{(1)},\tau^{(2)},\ldots))\to \mathcal P^{-1}(\Av(\tau'^{(1)},\tau'^{(2)},\ldots)).\]  As stressed in \cite{DefantFertilityWilf}, this is a very strong condition. 

Throughout this article, we are interested in joint equidistribution of permutation statistics. The following definition formalizes this notion. 

\begin{definition}\label{Def6} 
A \emph{permutation statistic} is a function from the set of normalized permutations to $\mathbb C$. Let $A$ and $A'$ be sets of normalized permutations, and let $\mathcal E$ be a set of permutation statistics. We say that the elements of $\mathcal E$ are \emph{jointly equidistributed on} $A$ \emph{and} $A'$ if there is a bijection $g:A\to A'$ such that $f(g(\pi))=f(\pi)$ for all $\pi\in A$ and all $f\in\mathcal E$. 
\end{definition} 

A \emph{descent} of a permutation $\pi=\pi_1\cdots\pi_n$ is an index $i\in[n-1]$ such that $\pi_i>\pi_{i+1}$. The \emph{descent set} of $\pi$, denoted $\Des(\pi)$, is the set of descents of $\pi$. Let $\LenDes$ denote the set of all permutation statistics $f$ such that $f(\pi)$ only depends on the length and descent set of $\pi$. A few such statistics (see \cite{Bouvel} or \cite{Claesson2} for their definitions) are $\des,\text{asc},\text{maj},\text{valley},\peak,\text{ddes},\text{dasc},\text{rir},\text{rdr},\text{lir},\text{ldr}$.
It is straightforward to show that an index $i$ is a descent of $\pi$ if and only if the vertex whose label is read $i^\text{th}$ in the in-order traversal of $\mathcal I^{-1}(\pi)$ has a right child. This means that the length and descent set of a permutation are determined by the permutation's skeleton.

\begin{definition}\label{Def1}
We say a permutation statistic $f$ is \emph{skeletal} if for every permutation $\pi$, $f(\pi)$ only depends on the skeleton of $\pi$.
\end{definition}

There are several important skeletal statistics that are not in $\LenDes$. A few examples (see \cite{Bouvel} or \cite{Claesson2} for their definitions) are $\rmax,\lmax,\indmax,\slmax,\slmax\circ\rev$. Theorem 3.2 in \cite{Bouvel} gives a list of several permutation statistics and states that the statistics in that list are jointly equidistributed on $s^{-1}(\Av_n(231))$ and $s^{-1}(\Av_n(132))$ for every $n\geq 1$. The statistics in the list appear somewhat arbitrary at first, but Theorem \ref{Thm2} below clarifies this matter, showing that all but one of those statistics appear in the list precisely because they are skeletal. The one remaining statistic is interesting; it appears in the list for a slightly different reason. 

We will see that the polyurethane action allows us to understand joint equidistribution of statistics on ``higher-order" preimages of permutations under $s$. In order to make this more precise, we make the following definition, which is motivated by the conjecture of Bouvel and Guibert mentioned above (which is stated below in \eqref{Eq4}). 

\begin{definition}\label{Def4}
Let $\rev$ denote the reversal operator defined on permutations by $\rev(\pi_1\cdots\pi_n)=\pi_n\cdots\pi_1$. A \emph{higher-order twisted stack-sorting operator} is a map of the form $\mathfrak s=\nu_m\circ\nu_{m-1}\circ\cdots\circ\nu_1$, where $\nu_1,\ldots,\nu_m\in\{s,\rev\}$ and $\nu_1=s$. 
\end{definition}

\section{Polyurethane Actions and Some Applications}\label{Sec:Main}

In the introduction, we defined the toggles $p_1,\ldots,p_{n-1}\in\mathfrak S_{\mathcal L( P)}$, where $ P$ is a poset with $n$ elements and $\mathcal L( P)$ is the set of all linear extensions of $ P$. We then said that we could view a binary plane tree as the Hasse diagram of a poset and use the in-order reading $\mathcal I$ to transfer these toggles so that they are defined on permutations. It will be convenient to have an equivalent definition of these maps that avoids any reference to linear extensions of posets. For $\pi\in S_n$ and $i\in[n-1]$, let $\swap_i(\pi)$ be the permutation obtained from $\pi$ by swapping the positions of $i$ and $i+1$. Similarly, for each decreasing plane tree $T$ on $[n]$, let $\swap_i(T)$ be the labeled tree obtained from $T$ by swapping the labels $i$ and $i+1$ (the resulting tree is not necessarily a decreasing plane tree). 

\begin{definition}\label{Def5}
Define the \emph{polyurethane toggle} $p_i\in\mathfrak S_{S_n}$ by \[p_i(\pi)=\begin{cases} \swap_i(\pi), & \mbox{if } \text{there is an entry }a>i+1\text{ appearing between }i\text{ and }i+1\text{ in }\pi; \\ \pi, & \mbox{otherwise.}  \end{cases}\] Let $\mathscr P_n=\langle p_1,\ldots,p_{n-1}\rangle$ be the subgroup of $\mathfrak S_{S_n}$ generated by $p_1,\ldots,p_{n-1}$. We call $\mathscr P_n$ the $n^\text{th}$ \emph{polyurethane group}.  
\end{definition}

For example, $p_2(3547126)=\swap_2(3547126)=2547136$ because the entry $4$ is larger than $3$ and lies between $2$ and $3$ in $3547126$. Note that $p_{n-1}$ is just the identity element of $\mathfrak S_{S_n}$. The elements of $\mathscr P_n$ preserve the skeletons of the permutations on which they act. In other words, any two permutations in $S_n$ that lie in the same $\mathscr P_n$-orbit must have the same skeleton. The following theorem shows that the converse is also true.  

\begin{theorem}\label{Thm1}
Two permutations in $S_n$ are in the same orbit under the action of $\mathscr P_n$ if and only if they have the same skeleton. The number of orbits of the $\mathscr P_n$-action on $S_n$ is the $n^\text{th}$ Catalan number $C_n=\frac{1}{n+1}{2n\choose n}$. Every orbit contains a unique $231$-avoiding permutation and a unique $132$-avoiding permutation.   
\end{theorem}
\begin{proof}
Suppose $\pi\in S_n$ contains the pattern $231$. It is straightforward to check that there must be some $i\in[n-1]$ and some $a\in[n]$ with $a>i+1$ such that $i+1$ appears to the left of $a$ in $\pi$ and $i$ appears to the right of $a$ in $\pi$. The permutation $p_i(\pi)$ contains strictly fewer copies of the pattern $231$ than $\pi$ does. This shows that we can repeatedly apply the polyurethane toggles until we eventually reach a $231$-avoiding permutation. Hence, every orbit of the $\mathscr P_n$-action contains at least one $231$-avoiding permutation. A similar argument shows that every orbit contains at least one $132$-avoiding permutation. We know from \eqref{Eq1} that there are $C_n$ $231$-avoiding permutations (and $C_n$ $132$-avoiding permutations) in $S_n$, so there are at most $C_n$ orbits. It is well known that $C_n$ is the number of (unlabeled) binary plane trees on $n$ vertices, so it is the number of skeletons of permutations in $S_n$. We saw above that any two permutations in the same orbit must have the same skeleton, so there are at least $C_n$ orbits. This proves that there are \emph{exactly} $C_n$ orbits. It follows that two permutations in $S_n$ with the same skeleton must be in the same orbit. Furthermore, each orbit contains a \emph{unique} $231$-avoiding permutation and a \emph{unique} $132$-avoiding permutation.  
\end{proof}

Many of the results in \cite{DefantCatalan} and \cite{DefantFertilityWilf} depend on ``sliding operators" $\swu, \swd:S_n\to S_n$. Those two articles give different equivalent definitions of these maps. We can give a third definition of $\swu$ and $\swd$ with the help of Theorem \ref{Thm1}; it is straightforward to check that the following definition is equivalent to the ones presented in \cite{DefantCatalan} and \cite{DefantFertilityWilf}. 

\begin{definition}\label{Def2}
Given $\pi\in S_n$, let $\swu(\pi)$ be the unique $132$-avoiding permutation with the same skeleton as $\pi$. Let $\swd(\pi)$ be the unique $231$-avoiding permutation with the same skeleton as $\pi$. 
\end{definition}

\begin{remark}\label{Rem1}
We can restrict the maps $\swu$ and $\swd$ to $\Av_n(231)$ and $\Av_n(132)$, respectively. It is clear that $\swu:\Av_n(231)\to\Av_n(132)$ and $\swd:\Av_n(132)\to\Av_n(231)$ are inverse bijections that preserve skeletons. 
\end{remark}

Settling a conjecture of Claesson, Dukes, and Steingrimsson, Bouvel and Guibert \cite{Bouvel} proved that the permutation classes $\Av(231)$ and $\Av(132)$ are fertility Wilf equivalent, meaning that $|s^{-1}(\Av_n(231))|=|s^{-1}(\Av_n(132))|$ for every $n\geq 1$. In fact, they proved the much stronger assertion that the permutation statistics in the set 
\begin{equation}\label{Eq8}
\LenDes\cup\{\rmax,\lmax,\zeil,\indmax,\slmax,\slmax\circ\rev\}
\end{equation} 
are jointly equidistributed on $s^{-1}(\Av_n(231))$ and $s^{-1}(\Av_n(132))$ for every $n\geq 1$ (see \cite{Bouvel} for the definitions of these statistics). They also conjectured that 
\begin{equation}\label{Eq4}
|\mathfrak s^{-1}(\Av_n(231))|=|\mathfrak s^{-1}(\Av_n(132))|\text{ for every higher-order twisted stack-sorting operator }\mathfrak s
\end{equation} (see Definition \ref{Def4}), and they suggested that the statistics in \eqref{Eq8} might also be jointly equidistributed on $\mathfrak s^{-1}(\Av_n(231))$ and $\mathfrak s^{-1}(\Av_n(132))$. One especially notable statistic appearing in \eqref{Eq8} is the \emph{Zeilberger statistic} $\zeil$, which originated in Zeilberger's study of $2$-stack-sortable permutations \cite{Zeilberger} and has received attention in subsequent articles such as \cite{Bousquet98, Bouvel, Claesson2, DefantFertilityWilf}. For $\pi\in S_n$, $\zeil(\pi)$ is defined to be the largest integer $m$ such that the entries $n,n-1,\ldots,n-m+1$ appear in decreasing order in $\pi$. All of the statistics in \eqref{Eq8} except $\zeil$ are skeletal. 

Bouvel and Guibert gave a somewhat complicated proof of the above equidistribution result using generating trees. In Theorem~\ref{Thm2} below, we obtain a more general version of their theorem and prove their conjecture via a simple application of the polyurethane action. We first need the following lemma, which is the primary reason why the polyurethane toggles are useful for studying the stack-sorting map. 

\begin{lemma}\label{Lem1}
If $\sigma\in S_n$ and $i\in[n-1]$ are such that $p_i(s(\sigma))=\swap_i(s(\sigma))$, then $p_i(\sigma)=\swap_i(\sigma)$ and $s(p_i(\sigma))=p_i(s(\sigma))$.
\end{lemma} 
\begin{proof}
If $p_i(s(\sigma))=\swap_i(s(\sigma))$, then there must be an entry lying between $i$ and $i+1$ in $s(\sigma)$ that is larger than $i+1$. Since $s(\sigma)=\mathcal P(\mathcal I^{-1}(\sigma))$ by \eqref{Eq6}, it follows from the definition of the postorder traversal and the fact that $\mathcal I^{-1}(\sigma)$ is a decreasing plane tree that $i$ is not a descendant of $i+1$ in $\mathcal I^{-1}(\sigma)$. This means that there is an entry lying between $i$ and $i+1$ in $\sigma$ that is larger than $i+1$, so $p_i(\sigma)=\swap_i(\sigma)$. We have $\mathcal I^{-1}(p_i(\sigma))=\swap_i(\mathcal I^{-1}(\sigma))$, so it follows from \eqref{Eq6} and the definition of the postorder traversal that \[s(p_i(\sigma))=\mathcal P(\mathcal I^{-1}(p_i(\sigma)))=\swap_i(\mathcal P(\mathcal I^{-1}(\sigma)))=\swap_i(s(\sigma))=p_i(s(\sigma)). \qedhere\]  
\end{proof}
\begin{theorem}\label{Thm2}
If $\mathfrak s$ is a higher-order twisted stack-sorting operator and $\pi,\pi'\in S_n$ have the same skeleton, then $\zeil$ and all of the skeletal statistics are jointly equidistributed on $\mathfrak s^{-1}(\pi)$ and $\mathfrak s^{-1}(\pi')$. In particular, $\zeil$ and all of the skeletal statistics are jointly equidistributed on $\mathfrak s^{-1}(\Av_n(231))$ and $\mathfrak s^{-1}(\Av_n(132))$.
\end{theorem}

\begin{proof}
Theorem \ref{Thm1} tells us that $\pi$ and $\pi'$ are in the same $\mathscr P_n$-orbit. Since $\mathscr P_n$ is generated by the polyurethane toggles $p_1,\ldots,p_{n-1}$, it suffices to prove the first statement of the theorem in the case in which $\pi'=p_i(\pi)$ for some $i\in[n-1]$. The proof is trivial if $\pi=\pi'$, so we can assume $\pi'=\swap_i(\pi)$. 

According to Definition \ref{Def4}, $\mathfrak s=\nu_m\circ\nu_{m-1}\circ\cdots\circ\nu_1$ for some $\nu_1,\ldots,\nu_m\in\{s,\rev\}$ with $\nu_1=s$. Let $\varphi_0:S_n\to S_n$ be the identity map. For $j\in\{1,\ldots,m\}$, let $\varphi_j=\varphi_{j-1}\circ\nu_{m+1-j}$. We prove by induction on $j$ that $p_i(\varphi_j^{-1}(\pi))=\varphi_j^{-1}(\pi')$ for all $j\in\{0,1,\ldots,m\}$. The base case $j=0$ is trivial, so assume that $j\in\{1,\ldots,m\}$ and that $p_i(\varphi_{j-1}^{-1}(\pi))=\varphi_{j-1}^{-1}(\pi')$. It is easy to see that the involutions $p_i,\rev\in\mathfrak S_{S_n}$ commute. Therefore, if $\nu_{m+1-j}=\rev$, we have \[p_i(\varphi_j^{-1}(\pi))=p_i(\rev(\varphi_{j-1}^{-1}(\pi)))=\rev(p_i(\varphi_{j-1}^{-1}(\pi)))=\rev(\varphi_{j-1}^{-1}(\pi'))=\varphi_j^{-1}(\pi')\] as desired. Now assume $\nu_{m+1-j}=s$. Choose $\sigma\in\varphi_j^{-1}(\pi)$. We have $s(\sigma)\in\varphi_{j-1}^{-1}(\pi)$, so $p_i(s(\sigma))\in\varphi_{j-1}^{-1}(\pi')$. Since $\pi\neq\pi'$, this implies that $s(\sigma)\neq p_i(s(\sigma))$. By the definition of $p_i$, we must have $p_i(s(\sigma))=\swap_i(s(\sigma))$. We can now use Lemma \ref{Lem1} to see that $s(p_i(\sigma))=p_i(s(\sigma))\in\varphi_{j-1}^{-1}(\pi')$. Thus, $p_i(\sigma)\in s^{-1}(\varphi_{j-1}^{-1}(\pi'))=\varphi_j^{-1}(\pi')$. As $\sigma$ was arbitrary, this proves that $p_i(\varphi_j^{-1}(\pi))\subseteq\varphi_j^{-1}(\pi')$. Since $\pi=p_i(\pi')$, we can use the same argument with the roles of $\pi$ and $\pi'$ interchanged to prove the reverse containment. This completes the inductive step. In the following paragraph, we will make use of the fact, which we just proved, that 
\begin{equation}\label{Eq2}
s(p_i(\sigma))=p_i(s(\sigma))\text{ whenever }\sigma\in\varphi_j^{-1}(\pi)\text{ and }\nu_{m+1-j}=s.
\end{equation}

Now that we have proven that $p_i(\varphi_j^{-1}(\pi))=\varphi_j^{-1}(\pi')$ for all $j\in\{0,1,\ldots,m\}$, we can set $j=m$ to see that $p_i(\mathfrak s^{-1}(\pi))=\mathfrak s^{-1}(\pi')$. We know that $f(p_i(\sigma))=f(\sigma)$ for every $\sigma\in\mathfrak s^{-1}(\pi)$ and every skeletal statistic $f$. In order to complete the proof of the first statement of the theorem, we need to show that $\zeil(p_i(\sigma))=\zeil(\sigma)$ for every $\sigma\in\mathfrak s^{-1}(\pi)$. For this, we appeal to Lemma 3.1 in \cite{DefantFertilityWilf}, which states that $\zeil(\lambda)=\min\{\rmax(\lambda),\tl(s(\lambda))\}$ for every $\lambda\in S_n$. Here, $\tl$ is the ``tail length" statistic and $\rmax(\lambda)$ is the number of right-to-left maxima of $\lambda$. For the purposes of this proof, we only need the fact that $\tl$ and $\rmax$ are skeletal statistics (see \cite{DefantFertilityWilf} for more details). Choose $\sigma\in\mathfrak s^{-1}(\pi)$. Since $\rmax$ is skeletal, $\rmax(p_i(\sigma))=\rmax(\sigma)$. Noting that $\mathfrak s=\varphi_m$ and $\nu_1=s$ (by Definition \ref{Def4}), we can use \eqref{Eq2} to see that $s(p_i(\sigma))=p_i(s(\sigma))$. We now use the fact that $\tl$ is skeletal to see that $\tl(s(p_i(\sigma)))=\tl(p_i(s(\sigma)))=\tl(s(\sigma))$. This proves that \[\zeil(p_i(\sigma))=\min\{\rmax(p_i(\sigma)),\tl(s(p_i(\sigma)))\}=\min\{\rmax(\sigma),\tl(s(\sigma))\}=\zeil(\sigma).\] 

Recall from Remark \ref{Rem1} that $\swu:\Av_n(231)\to\Av_n(132)$ is a skeleton-preserving bijection. For each $\pi\in\Av_n(231)$, we can use the first statement of the theorem to see that $\zeil$ and all of skeletal statistics are jointly equidistributed on $\mathfrak s^{-1}(\pi)$ and $\mathfrak s^{-1}(\swu(\pi))$. It follows that these statistics are jointly equidistributed on $\mathfrak s^{-1}(\Av_n(231))$ and $\mathfrak s^{-1}(\Av_n(132))$. 
\end{proof}

If we appeal to a result from \cite{DefantFertilityWilf}, we can extensively generalize the second part of Theorem \ref{Thm2}. 
Given $\pi\in S_n$, let 
\[\chi_m(\pi)=\begin{cases}  (n+m-1)\cdots(n+3)(n+1)\pi(n+2)(n+4)\cdots(n+m) & \mbox{if } m\equiv 0\pmod 2; \\  (n+m)\cdots(n+3)(n+1)\pi(n+2)(n+4)\cdots(n+m-1)& \mbox{if } m\equiv 1\pmod 2. \end{cases}\]
For example, 
\[\chi_5(132)=86413257,\quad\text{and}\quad \chi_6(132)=864132579.\]
Let 
\begin{equation}\label{Eq7}
\mathcal A=\bigcup_{m\geq 0}\{\chi_m(1),\chi_m(12),\chi_m(1423),\chi_m(2143)\}.
\end{equation} 
Let $\tau^{(1)},\tau^{(2)},\ldots$ be a (possibly empty) list of permutations taken from the set $\mathcal A$, and let $\tau'^{(i)}=\swu(\tau^{(i)})$ for all $i$. In \cite{DefantFertilityWilf}, it is proven that \[\swu(\Av(231,\tau^{(1)},\tau^{(2)},\ldots))=\Av(132,\tau'^{(1)},\tau'^{(2)},\ldots).\] Applying the first part of Theorem \ref{Thm2} to each $\pi\in\Av(231,\tau^{(1)},\tau^{(2)},\ldots)$, we obtain the following result. 

\begin{theorem}\label{Thm6}
Let $\mathfrak s$ be a higher-order twisted stack-sorting operator. If $\tau^{(1)},\tau^{(2)},\ldots$ is a list of permutations taken from the set $\mathcal A$ in \eqref{Eq7}, then $\zeil$ and all of the skeletal statistics are jointly equidistributed on $\mathfrak s^{-1}(\Av(\tau^{(1)},\tau^{(2)},\ldots))$ and $\mathfrak s^{-1}(\Av(\tau'^{(1)},\tau'^{(2)},\ldots))$. 
\end{theorem}

We now prove the surprising fact that the skeleton of a permutation $\pi$ determines the \emph{entire} set of rooted plane trees appearing as skeletons of trees in $\mathcal P^{-1}(\pi)$. The polyurethane action makes the proof remarkably simple. 

\begin{theorem}\label{Thm4}
If $\pi,\pi'\in S_n$ have the same skeleton, then there exists a skeleton-preserving bijection \[\omega:\mathcal P^{-1}(\pi)\to \mathcal P^{-1}(\pi').\] 
\end{theorem}

\begin{proof}
We know by Theorem \ref{Thm1} that $\pi$ and $\pi'$ are in the same $\mathscr P_n$-orbit. Because $\mathscr P_n$ is generated by $p_1,\ldots,p_{n-1}$, it suffices to prove the theorem in the case in which $\pi'=p_i(\pi)$ for some $i\in[n-1]$. The proof is trivial if $\pi=\pi'$, so we can assume $\pi'=\swap_i(\pi)$. Choose $T\in \mathcal P^{-1}(\pi)$. There must be an entry lying between $i$ and $i+1$ in $\pi$ that is larger than $i+1$. It follows from the definition of the postorder traversal and the fact that $T$ is a decreasing plane tree that $i$ is not a descendant of $i+1$ in $T$. This means that $\swap_i(T)$ is a decreasing plane tree. Furthermore, $\mathcal P(\swap_i(T))=\swap_i(\mathcal P(T))=\swap_i(\pi)=\pi'$. As $T$ was arbitrary, this shows that $\swap_i(\mathcal P^{-1}(\pi))\subseteq \mathcal P^{-1}(\pi')$. We can repeat this argument with the roles of $\pi$ and $\pi'$ interchanged to prove that $\swap_i(\mathcal P^{-1}(\pi))=\mathcal P^{-1}(\pi')$. We now put $\omega=\swap_i$ to complete the proof.   
\end{proof}

Most familiar skeletal permutation statistics (des, peak, maj, rmax, slmax, etc.) are easily seen to be skeletal. Theorem \ref{Thm2} provides us with several interesting skeletal statistics that are not at all obviously skeletal. One such statistic is the fertility statistic itself! We define this statistic by $\fer(\pi)=|s^{-1}(\pi)|$. More generally, if $\mathfrak s$ is a higher-order twisted stack-sorting operator, then we can define the \emph{higher-order twisted fertility statistic} $\fer_{\mathfrak s}$ by $\fer_{\mathfrak s}(\pi)=|\mathfrak s^{-1}(\pi)|$. Theorem \ref{Thm2} tells us that all of these statistics are skeletal. Another interesting statistic, which was introduced in \cite{DefantCatalan} in order to understand so-called \emph{uniquely sorted permutations}, is the \emph{deficiency statistic} $\de$. Bousquet-M\'elou \cite{Bousquet} defined a permutation to be \emph{sorted} if its fertility is positive. If $\pi\in S_n$, then $\de(\pi)$ is defined to be the smallest nonnegative integer $\ell$ such that $\pi(n+1)(n+2)\cdots(n+\ell)$ is sorted. Because we now know that the fertility statistic $\fer$ is skeletal, it is easy to verify that $\de$ is also skeletal.  

Bousquet-M\'elou \cite{Bousquet} defined a decreasing binary plane tree to be \emph{canonical} if every vertex $v$ that has a left child also has a nonempty right subtree $T_v^R$ such that the first entry in $\mathcal I(T_v^R)$ is smaller than the label of the left child of $v$. She defined a permutation $\pi$ to be canonical if $\mathcal I^{-1}(\pi)$ is canonical. She then proved the following result. 

\begin{theorem}[\hspace{-.01cm}\cite{Bousquet}]\label{Thm3}
For every sorted permutation $\pi$, there is a unique canonical permutation $\sigma\in s^{-1}(\pi)$. Moreover, the fertility of $\pi$ and the set of binary plane trees that are skeletons of elements of $s^{-1}(\pi)$ only depend on the skeleton of $\sigma$. 
\end{theorem}

Given a sorted permutation $\pi\in S_n$, we call the unique canonical permutation in $s^{-1}(\pi)$ the \emph{canonical preimage of }$\pi$. Theorem \ref{Thm3} led Bousquet-M\'elou to ask for a general method for computing the fertility of a permutation from the skeleton of its canonical preimage. This was accomplished in \cite{DefantPostorder,DefantPreimages,DefantClass} using different language. Invoking Theorem \ref{Thm2}, we obtain the following strengthening of Bousquet-M\'elou's theorem. 

\begin{corollary}\label{Cor2}
If $\mathfrak s$ is a higher-order twisted stack-sorting operator and $\pi,\pi'\in S_n$ are sorted permutations whose canonical preimages have the same skeleton, then all of the skeletal statistics are jointly equidistributed on $\mathfrak s^{-1}(\pi)$ and $\mathfrak s^{-1}(\pi')$.
\end{corollary}

One might ask if Theorem \ref{Thm3} actually tells us anything new. In other words, Corollary \ref{Cor2} would follow immediately from Theorem \ref{Thm2} (without the help of Theorem \ref{Thm3}) if we could show that the skeleton of the canonical preimage of a sorted permutation $\pi$ determines the skeleton of $\pi$. This turns out to be false. The sorted permutations $42135$ and $32145$ have different skeletons, but their canonical preimages $45231$ and $35241$ have the same skeleton. Thus, Corollary \ref{Cor2} applies when $\pi=42135$ and $\pi'=32145$ even though Theorem \ref{Thm2} does not apply in this case. This also explains why we did not include the statistic $\zeil$ in the collection of jointly equidistributed statistics in Corollary \ref{Cor2}. We have $s^{-1}(42135)=\{45231\}$ and $s^{-1}(32145)=\{35241\}$, and $\zeil$ is not equidistributed on these two sets because $\zeil(45231)=1\neq 2=\zeil(35241)$. Let us remark, however, that we \emph{can} add $\zeil$ to the collection of jointly equidistributed statistics in Corollary~\ref{Cor2} \emph{if} $\mathfrak s\not\in\{s,\rev\circ s\}$. Indeed, in this case, we can write $\mathfrak s=\widehat{\mathfrak s}\circ\widetilde{\mathfrak s}$ for some higher-order twisted stack-sorting operators $\widehat{\mathfrak s}$ and $\widetilde{\mathfrak s}$. Theorem~\ref{Thm3} tells us that there is a skeleton-preserving bijection from $\widehat{\mathfrak s}^{-1}(\pi)$ to $\widehat{\mathfrak s}^{-1}(\pi')$, so it follows from Theorem~\ref{Thm2} that there is a skeleton-preserving bijection from $\mathfrak s^{-1}(\pi)=\widetilde{\mathfrak s}^{-1}(\widehat{\mathfrak s}^{-1}(\pi))$ to $\mathfrak s^{-1}(\pi')=\widetilde{\mathfrak s}^{-1}(\widehat{\mathfrak s}^{-1}(\pi'))$ that also preserves the $\zeil$ statistic. 

We have just seen that the skeleton of the canonical preimage of a sorted permutation $\pi$ does not determine the skeleton of $\pi$. The reverse dependency, however, does hold. 

\begin{theorem}\label{Thm5}
If $\pi,\pi'\in S_n$ are sorted permutations that have the same skeleton, then the canonical preimages of $\pi$ and $\pi'$ have the same skeleton. 
\end{theorem}

\begin{proof}
As in the proofs of Theorems \ref{Thm2} and \ref{Thm4}, it suffices to consider the case in which $\pi'=p_i(\pi)=\swap_i(\pi)$ for some $i\in[n-1]$. Let $\sigma$ and $\sigma'$ be the canonical preimages of $\pi$ and $\pi'$, respectively. Lemma \ref{Lem1} tells us that $s(p_i(\sigma))=p_i(s(\sigma))=p_i(\pi)=\pi'$. We claim that $p_i(\sigma)$ is canonical. If we can prove this, then we will know that $p_i(\sigma)=\sigma'$ because the canonical preimage of $\pi'$ is unique. This will prove that $\sigma$ and $\sigma'$ have the same skeleton, as desired. 

To prove the claim, we assume by way of contradiction that $p_i(\sigma)$ is not canonical. From this assumption, one can verify that $i+1$ is the left child of a vertex $v$ in $\mathcal I^{-1}(\sigma)$ and that $v$ has a nonempty right subtree $T_v^R$ such that the first entry of $\mathcal I(T_v^R)$ is $i$. One can now check that every entry of $\mathcal P(\mathcal I^{-1}(\sigma))$ appearing between $i$ and $i+1$ is a descendant of $i$ in $\mathcal I^{-1}(\sigma)$. Every such entry is necessarily smaller than $i$. Since $\mathcal P(\mathcal I^{-1}(\sigma))=s(\sigma)=\pi$ by \eqref{Eq6}, we find that there are no entries between $i$ and $i+1$ in $\pi$ that are greater than $i+1$. However, this means that $p_i(\pi)=\pi\neq\swap_i(\pi)$, which is a contradiction.     
\end{proof}

We end this section with a discussion of a somewhat unexpected connection among skeletons of permutations, the stack-sorting map, and the group structure of $S_n$. One can naturally identify $S_n$ with the symmetric group $\mathfrak S_{[n]}$ by associating $\pi=\pi_1\cdots\pi_n\in S_n$ with the bijection from $[n]$ to $[n]$ that sends $i$ to $\pi_i$ for all $i\in[n]$. This defines a group operation $\boldsymbol{\cdot}$ on $S_n$. More precisely, if $\pi=\pi_1\cdots\pi_n$ and $\sigma=\sigma_1\cdots\sigma_n$, then $\pi\boldsymbol{\cdot}\sigma=\pi_{\sigma_1}\cdots\pi_{\sigma_n}$. Let $\pi^{-1}$ denote the inverse of $\pi$ in the group $S_n$. Note that $\pi\in\Av(231)$ if and only if $\pi^{-1}\in\Av(312)$. 

One can show (see Exercise 19 in Chapter 8 of \cite{Bona}) that 
\begin{equation}\label{Eq3}
\pi,\sigma\in S_n\text{ have the same skeleton if and only if }\pi^{-1}\boldsymbol{\cdot} s(\pi)=\sigma^{-1}\boldsymbol{\cdot} s(\sigma).
\end{equation} Exercise 21 in Chapter 8 of \cite{Bona} asks the reader to compute the size of the set $\{\pi^{-1}\boldsymbol{\cdot} s(\pi):\pi\in S_n\}$. The answer is the $n^\text{th}$ Catalan number $C_n$. In fact, we can use Theorem \ref{Thm1} to obtain the following proposition. 

\begin{proposition}\label{Prop1}
We have \[\{\pi^{-1}\boldsymbol{\cdot} s(\pi):\pi\in S_n\}=\Av_n(312)\quad\text{and}\quad\{s(\pi)^{-1}\boldsymbol{\cdot}\pi:\pi\in S_n\}=\Av_n(231).\]
\end{proposition}
\begin{proof}
We know by Theorem \ref{Thm1} that every $\mathscr P_n$-orbit of $S_n$ contains a unique $231$-avoiding permutation. Since $s^{-1}(123\cdots n)=\Av_n(231)$ by \eqref{Eq1}, it follows from the above discussion that \[\{\pi^{-1}\boldsymbol{\cdot} s(\pi):\pi\in S_n\}=\{\pi^{-1}\boldsymbol{\cdot} s(\pi):\pi\in \Av_n(231)\}=\{\pi^{-1}:\pi\in \Av_n(231)\}=\Av_n(312).\] Therefore, \[\{s(\pi)^{-1}\boldsymbol{\cdot}\pi:\pi\in S_n\}=\{(\pi^{-1}\boldsymbol{\cdot} s(\pi))^{-1}:\pi\in S_n\}=\{\sigma^{-1}:\sigma\in\Av_n(312)\}=\Av_n(231). \qedhere\] 
\end{proof}

In Definition \ref{Def2}, we defined a sliding operator $\swd$. There are two alternative definitions of this operator (which better explain the name ``sliding operator" and the symbol ``$\swd$") appearing in \cite{DefantCatalan} and \cite{DefantFertilityWilf}. As a consequence of Proposition \ref{Prop1}, we obtain a fourth description of this operator that bears very little resemblance to the other three. 

\begin{corollary}\label{Cor3}
For every $\pi\in S_n$, we have $\swd(\pi)=s(\pi)^{-1}\boldsymbol{\cdot}\pi$. 
\end{corollary} 

\begin{proof}
We know by Proposition \ref{Prop1} that $s(\pi)^{-1}\boldsymbol{\cdot} \pi\in\Av_n(231)$, so it follows from \eqref{Eq1} that \linebreak$s(s(\pi)^{-1}\boldsymbol{\cdot} \pi)=123\cdots n$. This shows that $(s(\pi)^{-1}\boldsymbol{\cdot} \pi)^{-1}\boldsymbol{\cdot}s(s(\pi)^{-1}\boldsymbol{\cdot} \pi)=\pi^{-1}\boldsymbol{\cdot} s(\pi)$, so \eqref{Eq3} tells us that $\pi$ and $s(\pi)^{-1}\boldsymbol{\cdot}\pi$ have the same skeleton. The proof now follows from Definition \ref{Def2}.  
\end{proof}

\section{Open Problems}

It would be interesting to determine the actual isomorphism types (or even just the orders) of the polyurethane groups $\mathscr P_n$. We know that $p_{n-1}=1$ and that $p_i^2=1$ for each $i\in[n-2]$. It is also easy to see that $p_ip_j=p_jp_i$ when $i$ and $j$ are not consecutive integers, and one can verify that $(p_ip_{i+1})^6=1$. This shows that $\mathscr P_n$ is a quotient of the Coxeter group with Coxeter graph $\underbrace{\includegraphics[width=5cm]{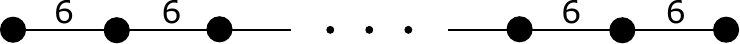}}_{n-2}$. It is easy to see that $\mathscr P_3$ has order $2$. One can also show that $\mathscr P_4$ is the Coxeter group $G_2$, which is isomorphic to $D_{12}$, the dihedral group of order $12$. 

We know by Theorem \ref{Thm2} that the statistic $\fer_{s^2}$ given by $\fer_{s^2}(\pi)=|s^{-2}(\pi)|$ is skeletal. It would be very interesting (and probably very useful) to have a method for determining $\fer_{s^2}(\pi)$ from the skeleton of $\pi$. Alternatively, one could attempt to follow the ideas introduced in \cite{DefantPostorder,DefantPreimages,DefantClass} to produce a method for determining $\fer_{s^2}(\pi)$ from the skeleton of the canonical preimage of $\pi$ when $\pi$ is sorted (those articles are phrased in terms of ``canonical valid hook configurations" instead of canonical preimages).   

\section{Acknowledgments}

The author thanks the anonymous referee for several helpful suggestions. He thanks Caleb Ji for a useful conversation about group theory and thanks Roktim Barkakati for a useful conversation about polyurethane. He also thanks Amanda Burcroff, who referred him to the article \cite{Bjorner}. The author was supported by a Fannie and John Hertz Foundation Fellowship and an NSF Graduate Research Fellowship.

\end{document}